\documentclass[a4paper,10pt,reqno]{amsart}
\usepackage{xcolor}
\usepackage{graphicx}
\usepackage{amsmath}
\usepackage{amssymb}
\usepackage{comment}

\theoremstyle{plain}
\newtheorem{theorem}{Theorem}[section]

\newtheorem{lemma}[theorem]{Lemma}

\newtheorem{remark}[theorem]{Remark}

\newcommand{\C}{\mathbb{C}}
\newcommand{\N}{\mathbb{N}}

\newcommand{\R}{\mathbb{R}}

\newcommand{\dC}{\mathbb{C}}

\newcommand{\dR}{\mathbb{R}}

\newcommand{\dZ}{\mathbb{Z}}

\newcommand{\cB}{\mathcal{B}}

\newcommand{\cN}{\mathcal{N}}

\newcommand{\eps}{\varepsilon}

\let\altphi\phi
\let\phi\varphi
\let\varphi\altphi
\let\altphi\undefined

\begin{document}
	
	\title[Nilpotent Fundamental Groups and Positive Ricci Curvature]{Compact Manifolds with Unbounded Nilpotent Fundamental Groups and Positive Ricci Curvature}
	
	\author{Elia Bruè, Aaron Naber, and Daniele Semola}
	
	\maketitle{}
	
	\begin{abstract}
	It follows from the work of Kapovitch and Wilking that a closed manifold with nonnegative Ricci curvature has a uniformly almost nilpotent fundamental group.  Leftover questions and conjectures, see \cite{KapovitchWilking} and \cite{PanRong}, have asked if in this context the fundamental group is actually uniformly almost abelian. 
	
	The main goal of this work is to construct examples $(M^{9}_k, g_k)$ with uniformly positive Ricci curvature ${\rm Ric}_{g_k}\geq 8$ whose fundamental groups 
	cannot be uniformly virtually abelian.
	\end{abstract}
	
\section{Introduction}

A consequence of the work of Kapovitch and Wilking in \cite{KapovitchWilking} for a closed manifold $(M^n,g)$ with $\mathrm{Ric}\ge 0$ is that the fundamental group $\pi_1(M^n)$ has a nilpotent subgroup $N\leq \pi_1(M)$ of uniformly bounded index $[N,\pi_1(M)]\leq C(n)$.  An important open question from their work, see \cite[page 48]{KapovitchWilking}, is whether this nilpotent subgroup $N$ can be taken to be abelian.  More generally, it is asked whether for a space with $\mathrm{Ric}\geq -(n-1)$ the torsion of the local fundamental group\footnote{$\mathrm{Im}\Big(\pi_1(B_{\epsilon(n)}(p))\to \pi_1(B_1(p))\Big)$} can be taken to be uniformly almost abelian.

The question of Kapovitch and Wilking, and conjectures of Fukaya and Yamaguchi, have been quantitatively refined into conjecture by Pan and Rong, see \cite[Conjecture 2.22]{PanRong} and \cite[Conjecture 12]{SantosZamora}; namely, it is conjectured that compact spaces with nonnegative Ricci curvature have fundamental groups which are uniformly almost abelian. The results of \cite{MazurRongWang} and \cite{SantosZamora} are able to answer these conjectures in the affirmative under the additional hypothesis that the manifolds are noncollapsing; more precisely, the index of the abelian subgroup is bounded in terms of a constant depending on the volume of unit balls in the universal cover.\\

The main goal of this paper is to answer these questions in the general case with the following construction:\\

\begin{theorem}\label{thm:mainIntro}
  There exists a sequence of smooth $9$-dimensional Riemannian manifolds $(N_k^{9}, g_k)$ such that: 
  \begin{enumerate}
  \item  $\mathrm{Ric}_{g_k} \geq 8$, and consequently $\mathrm{diam} (N_k)\leq \mathrm{diam} (\widetilde N_k)\leq \pi$ with $\widetilde N_k$ the universal cover of $N_k$;
  \item The fundamental group $\pi_1(N_k)$ is a degree-two extension of the finite Heisenberg group $H_3(\mathbb{Z}/k\mathbb{Z})$ for each $k\in\mathbb{N}$.	
  \end{enumerate}
\end{theorem} 
\vspace{.3cm}

The interest of Theorem \ref{thm:mainIntro} stems from the fact that the groups $H_3(\mathbb{Z}/k\mathbb{Z})$, see Section \ref{sec:HeisenbergNil}, are not \textit{uniformly virtually abelian} with respect to $k$. That is, the index of every abelian subgroup $A_k \leq H_3(\mathbb{Z}/k\mathbb{Z})$ diverges as $k \to \infty$.  In particular, crossing our example with a flat torus we answer the question of Kapovitch and Wilking and conjectures of \cite[Conjecture 2.22]{PanRong} and \cite[Conjecture 12]{SantosZamora} in the negative for every dimension $n\ge 9$.

As such, Theorem \ref{thm:mainIntro} also refutes the natural counterpart for Ricci curvature of some long-standing conjectures of Fukaya and Yamaguchi for fundamental groups of manifolds with $\mathrm{sec}\ge 0$.  In the context of noncompact manifolds with torsion-free fundamental group, see also the excellent example of Wei \cite{Wei88}.  One is still left with the original conjecture of Fukaya and Yamaguchi in the context of nonnegative sectional curvature, which is very much open.  Additionally, if we compare to \cite{KPTcenter}, we see that the {\it torsion in the center} conjecture for spaces with nonnegative sectional curvature must fail under the weaker nonnegative Ricci assumption.  

Theorem \ref{thm:mainIntro} is also connected to another series of questions which remain open.  In \cite{BNS_Milnor} and \cite{BNS_Milnor6} examples of manifolds with nonnegative Ricci curvature and infinitely generated fundamental group were constructed.  The infinite generation was however abelian in nature. For instance, it is unknown if there always exists a normal abelian subgroup $A\leq \pi_1(M)$ such that the quotient $\pi_1(M)/A$ is finitely generated. \\

The main step in the construction of the spaces $N^9_k$ in Theorem \ref{thm:mainIntro} is based on the construction of simply connected four-manifolds $M^4_k$ with effective isometric actions by $H_3(\dZ/k\dZ)$.  These four-manifolds even have uniformly bounded Ricci curvature.  However, the actions of $H_3(\dZ/k\dZ)$ on the $M^4_k$ are not free.  We will construct the universal covers $\widetilde N^{9}_k$ from $M^4_k$ by looking at the spin bundles of $M^4_k$ and lifting the actions of $H_3(\dZ/k\dZ)$ to {\it free} actions.  In particular, this also explains the $\dZ/2\dZ$ extension of Theorem \ref{thm:mainIntro}. To reduce the dimension of the spin bundle, we take the quotient with respect to a circle action with totally geodesic fibers, following the suggestion of an anonymous referee, whom we thank.\\

Our precise construction is the following:\\

\begin{theorem}\label{thm:HeisenberIntro}
	For every $k \in \mathbb{N}$ there exists a closed, simply connected $4$-manifold $(M_k^4, g_k)$ with $\mathrm{diam}(M_k) \leq \pi$ and $0 < \mathrm{Ric}_{g_k} \leq 3$, admitting an effective isometric action of the finite Heisenberg group $H_3(\mathbb{Z}/k\mathbb{Z})$.
\end{theorem}

Each $M_k^4$ above is diffeomorphic to the connected sum of $(k-1)$ copies of $S^2\times S^2$. We refer the reader to Section \ref{sec:structure of M} for more details on this and on the geometry of the examples.

\begin{remark}
	Theorem \ref{thm:HeisenberIntro} is optimal in at least two senses:
	\begin{itemize}
		\item[i)] The family of finite groups $\Gamma$ that act smoothly and effectively on a sphere \(S^n\) for \(n = 2, 3\) is uniformly virtually abelian. See for instance \cite{DinkelbachLeeb}. In particular, there are no three dimensional examples as above.

		\item[ii)] In dimension four, the Fukaya-Yamaguchi conjecture has been proved in \cite{BNS_FY4d}. More generally, by \cite[Theorem 1.2]{Mundet19} and \cite{Gromovbetti}, there is no family of four dimensional examples as above with positive sectional curvature.  See \cite{BNS_FY4d} for the argument.
	\end{itemize}
\end{remark}

\vspace{.2cm}

\subsection*{Broader Mathematical Context}
We recall that the fundamental group of every closed manifold with $\mathrm{Ric}>0$ is finite, as a consequence of Myers' theorem.  It is not the case that one has uniform finiteness, the easiest examples being the Lens spaces in dimension three.  Similarly, one also knows that the fundamental group of every closed manifold with $\mathrm{Ric}\ge 0$ is virtually abelian, as a corollary of the Cheeger-Gromoll splitting theorem, see \cite[Theorem 3]{CheegerGromollRic}.  

On the other hand, in the noncompact case Wei constructed a family of complete $(M^n,g)$ with $\mathrm{Ric}\ge 0$ and for which $\pi_1$ is a nilpotent lattice \cite{Wei88}. Wilking used Wei's construction later in \cite{Wilking00} and proved that any finitely generated virtually nilpotent group is the fundamental group of some complete (noncompact) $(M^n,g)$ with $\mathrm{Ric}\ge 0$, for some $n\in\mathbb{N}$.  

For manifolds $(M^n,g)$ with nonnegative sectional curvature, fundamental groups are virtually abelian without any compactness assumption, by the Cheeger-Gromoll soul theorem \cite{CheegerGromollsoul} (combined again with the splitting theorem). Fukaya and Yamaguchi conjectured in \cite{FukayaYamaguchi} that the fundamental group of any $(M^n,g)$ with $\mathrm{sec}\ge 0$ should have an abelian subgroup with index bounded by a dimensional constant $C(n)$. Their conjecture remains open.\\ 

Kapovitch, Petrunin, and Tuschmann proved that fundamental groups of almost nonnegatively curved manifolds in any fixed dimension are uniformly virtually nilpotent in \cite{KPTAnn} after some earlier progress due to Fukaya and Yamaguchi in \cite{FukayaYamaguchi}. They conjectured that one should be able to arrange the finite-index nilpotent subgroup to have torsion contained in the center, see \cite[Main Conjecture 6.1.2]{KPTAnn}. If the ``Torsion in the center'' conjecture holds, then the Fukaya-Yamaguchi conjecture would be true as well, as observed by Wilking, see \cite[Section 6]{KPTAnn}.

We refer the reader to~\cite{Rong96,MazurRongWang,KPTcenter} for several partial results in the direction of the Fukaya--Yamaguchi conjecture; this list is not intended to be exhaustive.

\subsection*{Acknowledgments} Part of this work was conducted while D.S. was a Hermann Weyl Instructor. He is grateful to the FIM-Institute for Mathematical Research of ETH Z\"urich for the support and the excellent working conditions.

The authors are grateful to Shengxuan Zhou and to the anonymous reviewers for useful comments that helped improving a preliminary version of the note.

\section{Preliminaries}

\subsection{Heisenberg Groups}\label{sec:Heisenberg}
	The $3$-dimensional Heisenberg group over a commutative ring $(A,+,\cdot)$ is defined as
\begin{equation}
	H_3(A) := \left\lbrace 
	\left(\begin{matrix}
		1 & a & c
		\\
		0 & 1 & b
		\\
		0 & 0 & 1
	\end{matrix}\right)
	\, : \, a,b,c\in A
	\right\rbrace \, < \mathrm{GL}_3(A)\, .
\end{equation}
In the case that $(A,+)$ is generated by the identity $1\in A$, we have that $H_3(A)$ is generated as a group by the two elements
\begin{equation}
	X:= 	\left(\begin{matrix}
		1 & 1 & 0
		\\
		0 & 1 & 0
		\\
		0 & 0 & 1
	\end{matrix}\right) \, ,
	\quad 
	Y:= 	\left(\begin{matrix}
		1 & 0 & 0
		\\
		0 & 1 & 1
		\\
		0 & 0 & 1
	\end{matrix}\right) \, .
\end{equation}
If we set 
\begin{equation}
	Z:= 	\left(\begin{matrix}
		1 & 0 & 1
		\\
		0 & 1 & 0
		\\
		0 & 0 & 1
	\end{matrix}\right) \, ,
\end{equation}
then we have the relations
\begin{equation}
	XYX^{-1}Y^{-1} = Z \, , \quad XZ=ZX\, , \quad YZ = ZY \, .
\end{equation}	

We will mainly be interested in the cases $A=\R$, $A=\mathbb{Z}$, and $A=\mathbb{Z}/k\mathbb{Z}$ with $k\in\mathbb{N}$.

\subsection{Heisenberg Nilmanifolds and Group Actions}\label{sec:HeisenbergNil}

Fix $k\in \mathbb{Z}$.
We consider the free action of $\mathbb{Z}^2$ on $\R^2 \times S^1$ defined by
\begin{equation}
	(a,b)\cdot (x,y,z) := (x + a, y +  b, e^{- 2\pi i k a y}z)\, ,
\end{equation}
for every $(a,b)\in \mathbb{Z}^2$ and $(x,y,z)\in \R^2 \times S^1\subseteq \dR^2\times \dC$.
The quotient ${\rm Nil}_k^3$ is the so-called {\it Heisenberg nilmanifold of degree $k \in \mathbb{Z}$}. It admits a free, smooth $S^1$-action defined by 
\begin{equation}
	\theta \cdot [x,y,z] = [x,y,e^{i\theta} z] \, , \quad \theta \in S^1 \, ,
\end{equation}
where $[x,y,z]$ denotes the equivalence class of $(x,y,z) \in \mathbb{R}^2 \times S^1$. This $S^1$-action endows every Heisenberg nilmanifold ${\rm Nil}_k^3$ with the structure of a principal $S^1$-bundle over $T^2$, with projection map $\pi_k : {\rm Nil}_k^3 \to T^2$ induced by the natural map $\mathbb{R}^2 \times S^1 \ni (x,y,z) \mapsto (x,y) \in \R^2$.

\medskip

Assume that $k > 0$ for notational simplicity. A similar construction applies for $k < 0$.
 The nilmanifold ${\rm Nil}_k^3$ admits a free effective action of the finite Heisenberg group $H_3(\mathbb{Z}/k\mathbb{Z})$, whose generators act as
 \begin{equation}\label{eq:widetilderho}
	\begin{split}
		&X\cdot [x,y,z] := [x + 1/k, y, e^{- 2\pi i y}z] \, , \\
		&Y\cdot [x,y,z] := [x, y + 1/k, z] \, .
	\end{split}
\end{equation}
The commutator of the $X$ and $Y$ action gives the action by $Z$, which generates a cyclic group $\mathbb{Z}/k\mathbb{Z}$ acting by rotation on the $S^1$-fibers of ${\rm Nil}_k^3$:
 \begin{equation}
	\begin{split}
		&Z\cdot [x,y,z] := [x, y,e^{-2\pi i/k} z] \, .
	\end{split}
\end{equation}

Note that the actions of $X$ and $Y$ are lifted from the torus action
 \begin{equation}
	\begin{split}
		&X\cdot [x,y] := [x + 1/k, y] \, , \\
		&Y\cdot [x,y] := [x, y + 1/k] \, ,
	\end{split}
\end{equation}
where $[x,y] \in T^2 = \R^2/\mathbb{Z}^2$ denotes the equivalence class of $(x,y) \in \mathbb{R}^2$.\\

There are several equivalent ways to build the action of $H_3(\mathbb{Z}/k\mathbb{Z})$ on a nilmanifold ${\rm Nil}_k^3$.  The construction we have followed here has been exploited in the past, for instance in \cite{Zarhin} (resp. \cite{CsikosBalazsSzabo}) to construct an example of an irreducible complex surface (resp. smooth, closed $4$-manifold) such that the finite subgroups of the group of birational transformations (resp. diffeomorphisms) are not uniformly virtually abelian. See also the more recent \cite[Section 2]{CsikosBalazsSzabo2}.\\

\section{Structure of $M_k^4$}
\label{sec:structure of M}

In this section, we outline the diffeomorphic and metric structure of the manifolds $(M^4_k, g_k)$ built in Theorem \ref{thm:HeisenberIntro}.  We will go about the construction carefully and rigorously in the next sections.

\subsection{Sketch of Construction}

Our construction of $M_k^4$ is based on a variant of the Gibbons-Hawking ansatz (see for instance \cite{GibbonsHawking,AndersonKronheimerLebrun,LeBrun,GrossWilson}) with base space the three-sphere $S^3$. We consider two distinct circle fibers $F_0$ and $F_1$ of the Hopf fibration $S^3\to S^2$, viewed as the fibers corresponding to the north and south poles of $S^2$.  We remove $k$ points from each of the two chosen fibers, in a way so that discrete Hopf rotation by angle $2\pi/k$ leaves the set invariant.  There is a $\mathbb{Z}/k\mathbb{Z}\times \mathbb{Z}/k\mathbb{Z}\subseteq T^2$ action on $S^3$, which is a subaction of the usual torus action on $S^3\subset \C^2$, which maps this removed set to itself.  See Section \ref{Sec:Proofthm2} for more on this.
  
The next step is to construct the unique (up to orientation) $S^1$-bundle whose Chern class is invariant under this $\mathbb{Z}/k\mathbb{Z}\times \mathbb{Z}/k\mathbb{Z}$ action on $S^3$, and for which when restricted to each small $2$-sphere enclosing one of the removed points is the Hopf bundle. The resulting open four-manifold can be smoothly compactified by adding $2k$ points. The Riemannian metric $g_k$ is defined as (a small conformal perturbation of) the Gibbons-Hawking type metric \eqref{eq:h_k}, constructed using a Green's function with positive poles corresponding to the points removed from $F_0$ and negative poles corresponding to the points removed from $F_1$, see \eqref{eq:defu}. A key distinction from the classical Gibbons-Hawking construction is the behavior of the Green's function, which changes sign: it is positive near the fiber $F_0$ and negative near $F_1$.  This causes a great deal of delicacy and care in the region between the fibers.  We will show this $\mathbb{Z}/k\mathbb{Z}\times \mathbb{Z}/k\mathbb{Z}$ action on $S^3$ lifts to an (isometric) action by the Heisenberg group $H_3(\dZ/ k\dZ)$ on $M^4_k$, whose isotropy points are the $2k$ removed points.\\

Let us now discuss several viewpoints from which to understand the geometry of our $M^4_k$:

\subsection{$M^4_k$ as a Singular $S^1$ Bundle over $S^3$}\label{subsec:sinS1bundle}

The simply connected manifold $(M^4_k, g_k)$ admits an isometric $S^1$-action with $2k$ fixed points corresponding to the points that we removed from the Hopf fibers $F_0$ and $F_1$. The quotient space of the action is naturally homeomorphic to $S^3$, essentially by construction, and we shall denote $\pi : M_k^4 \to S^3$ to be the natural projection. Smooth $S^1$ actions with finitely many fixed points on closed simply connected $4$-manifolds are completely classified up to equivariant diffeomorphism. In particular, by \cite[Theorem 1.6]{ChurchLamotke}, we deduce that $M_k^4$ is diffeomorphic to the connected sum of $k-1$ copies of $S^2 \times S^2$. Hence, $M^4_k$ is spin and we can compute the homology as $H_1(M_k^4, \mathbb{Z}) = H_3(M_k^4, \mathbb{Z}) = 0$, and $H_2(M^4_k, \mathbb{Z}) \cong \mathbb{Z}^{2k-2}$.

\subsection{$M^4_k$ as a Singular $T^2$ Bundle Over $S^2$}
The map $\pi_{T^2} : M^4_k \to S^2$, obtained by composing $\pi : M_k^4 \to S^3$ and the Hopf map $\pi_{\rm Hopf} : S^3 \to S^2$, is a smooth torus fibration away from the preimages of the north and south poles in $S^2$. For every $x \in S^2$ different from the north and south pole, the fiber $\pi_{T^2}^{-1}(x)$ is diffeomorphic to a two-torus. On the other hand, the preimages of the south and north poles are homeomorphic to singular Kodaira fibers of type $I_k$. Topologically, the latter can be viewed as a chain of $k$ spheres $S_1^2, S_2^2, \dots, S_k^2$ with the property that  $S_i^2$ intersects $S_{i+1}^2$ at exactly one point, forming a loop (where $S_{k+1}^2$ is identified with $S_1^2$). 

The reader is referred to \cite{GrossWilson} and \cite{ChenViaclovskyZhang} for more on the appearance of singular Kodaira fibers through constructions based on the Gibbons-Hawking ansatz.

\subsection{$M^4_k$ as a Singular ${\rm Nil}_k^3$ Bundle Over $I$}
The map $\pi_{{\rm Nil}_k^3} : M^4_k \to [0,1]$, obtained by composing $\pi : M_k^4 \to S^3$ and  $S^3 \ni (z_1, z_2) \mapsto |z_1| \in [0,1]$, is a smooth fibration away from the preimages of $0$ and $1$. The (generic) fiber is diffeomorphic to the Heisenberg nilmanifold ${\rm Nil}_k^3$ of degree $k$.
The preimages of the boundary points are $I_k$ fibers. We refer the reader to \cite{HeinSunViaclovskyZhang} for further details on the emergence of fibrations with nilmanifold fibers via constructions based on the Gibbons-Hawking ansatz. Additionally, we point to the earlier works \cite{TianYau} and \cite{Hein} for the construction of Calabi-Yau metrics on manifolds that locally fiber as nilmanifolds over intervals.

The global action by the finite Heisenberg group $H_3(\mathbb{Z}/k\mathbb{Z})$ preserves the fibers of $\pi_{\text{Nil}_k^3}$. On the regular fibers this action matches the one described in Section \ref{sec:HeisenbergNil}. On the two singular $I_k$ fibers, the two generators of $H_3(\mathbb{Z}/k\mathbb{Z})$ act in the following ways:

\begin{enumerate}
	\item \textit{Cycling of the spheres}: The spheres $S_1^2, \dots, S_k^2$ are permuted cyclically, with the action $S_\ell^2 \to S_{\ell+1}^2$ for $\ell = 1, \dots, k-1$, and $S_k^2 \to S_1^2$.
	\item \textit{Asynchronous rotation}: The sphere $S_\ell^2$ is rotated by an angle $2\pi \ell / k$ (with respect to the axis connecting the touching points with the two neighbourhing $2$-spheres) for each $\ell = 1, \dots, k$.
\end{enumerate}

\subsection{Doubling of a Neighborhood of the $I_k$ Fiber}

Let $\Sigma \subset S^3$ be the Clifford torus. As we observe in the proof of Lemma \ref{lemma:invariance} below, the preimage $\pi^{-1}(\Sigma) \subset M_k^4$ is diffeomorphic to $\mathrm{Nil}_k^3$. It is not hard to see that the latter disconnects $M_k^4$ into two open regions $U_k^{\pm}$, each diffeomorphic to a neighborhood of an $I_k$ fiber.

Thus, $M_k^4$ can be obtained by gluing together these two neighbourhoods of the $I_k$ fiber along their boundaries $\partial U_k^{\pm}$, which are diffeomorphic to $\mathrm{Nil}_k^3$. Importantly, the gluing map is an equivariant diffeomorphism of $\mathrm{Nil}_k^3$ with respect to its $S^1$ and $H_3(\mathbb{Z}/k\mathbb{Z})$ symmetries. Moreover the gluing map switches the sign of the Chern class on the base two-torus. This is consistent with the fact that $M_k^4$ is simply connected, although the two neighbourhoods of the $I_k$ fibers deformation retract onto the singular fibers and, as such, have fundamental groups isomorphic to $\mathbb{Z}$. More in detail, we can apply Van Kampen's theorem to an open cover $M_k^4=V_k^+\cup V_k^{-}$ where $V_k^{\pm}\Supset U_k^{\pm}$ are tubular neighbourhoods. Thus $V_k^+\cap V_k^-$ is a tubular neighbourhood of the common boundaries $\partial U_k^{\pm}$ which deformation retracts onto $\partial U_k^{\pm}$. In particular, $V_k^+\cap V_k^-$ is homotopy equivalent to $\mathrm{Nil}_k^3$. With these choices, $\pi_1(V^{\pm}_k)\cong \pi_1(U_k^{\pm})\cong \pi_1(I_k)\cong\mathbb{Z}$. On the other hand, each generator of $\pi_1(V_k^+\cap V_k^-)\cong \pi_1(\mathrm{Nil}_k^3)$ is mapped to $0$ by at least one of the natural inclusions $(i_+)_*:\pi_1(V_k^+\cap V_k^-)\to \pi_1(V^+_k)$ and $(i_{-})_*:\pi_1(V_k^+\cap V_k^-)\to \pi_1(V^{-}_k)$. Thus we have
\begin{equation}
\pi_1(M^4_k)\cong \pi_1(V_k^+)\ast_{\pi_1(V_k^+\cap V_k^-)}\pi_1(V_k^-)\cong \{0\}\, .
\end{equation}
We also note that this point of view on the $M^4_k$'s can be exploited to compute their homologies and intersection forms without invoking the resuls from \cite{ChurchLamotke}. The intersection form could then be used to show that $M^4_k$ is homeomorphic to the connected sum of $(k-1)$ copies of $S^2\times S^2$. We avoid discussing further the details since \cite[Theorem 1.6]{ChurchLamotke} yields a stronger diffeomorphism statement.

\section{Proof of Theorem \ref{thm:HeisenberIntro}}\label{Sec:Proofthm2}

    The careful construction of $M^4_k$ begins with the three sphere $S^3$.  Let us view $S^3$ as the unit sphere in $\mathbb{R}^4=\C^2$. We consider the fibers $F_0,F_1\subset S^3$ of the Hopf fibration defined by
    \begin{equation}\label{eq:F_1F_2}
    	F_0 :=\{(z_1, 0): |z_1|=1\} \, , 
    	\quad
    	F_1 := \{(0, z_2): |z_2|=1\} \, .
    \end{equation}
    Let $k\in\mathbb{N}$ be fixed and $\mathcal{B}_k$ be the incomplete manifold obtained by removing the $k$-th roots of unity from the Hopf fibers $F_0$ and $F_1$. Namely, we set
    \begin{equation}
    	\mathcal{B}_k := S^3 \setminus \{(e^{2\pi i \ell/k},0), (0, e^{2\pi i \ell'/k})\, : \, 0\le \ell, \ell'\le k-1\}\, .
    \end{equation}
    Note that there is a natural action of $\mathbb{Z}/k \mathbb{Z} \times \mathbb{Z}/k  \mathbb{Z}$ on $\mathcal{B}_k$ given by

  \begin{equation}\label{eq:gamma_1gamma_2}
    	(\ell,\ell')\cdot(z_1,z_2) := (e^{2\pi i\ell/k}z_1,e^{2\pi i\ell'/k}z_2) \, .
    \end{equation}

The action is effective and isometric with respect to the restriction of the standard metric of $S^3$ to $\mathcal{B}_k$.

\medskip

We note that $H_2(\mathcal{B}_k, \mathbb{Z}) \simeq \mathbb{Z}^{2k-1}$ is generated by the homology classes of small $2$-spheres enclosing the points $p^0_\ell:=(e^{2\pi i \ell/k},0)$ and $p^1_{\ell'}:=(0, e^{2\pi i \ell'/k})$.  We enumerate them
\begin{align}
S^{\alpha}_\ell = \partial B_r(p^\alpha_\ell)\, ,
\end{align}
with $\alpha\in \{0,1\}$, $\ell\in\{0,\ldots,k-1\}$ and $r>0$ chosen sufficiently small so that the collection of balls $B_r(p^\alpha_\ell)$ is disjoint and each ball is diffeomorphic to a three ball.
\\

Let us now define the integral cohomology class $[\omega] \in H^2(\mathcal{B}_k, \mathbb{Z}) \simeq \mathbb{Z}^{2k-1}$ by the condition
\begin{align}\label{e:omega_def}
\int_{S^\alpha_\ell} \omega = (-1)^\alpha\, 2\pi\, ,	
\end{align}
where the orientation of each sphere $S_\ell^\alpha$ is induced by the standard orientation of $S^3\subset \R^4$.
Observe that $\omega$ in invariant under the $\mathbb{Z}/k \mathbb{Z} \times \mathbb{Z}/k  \mathbb{Z}$ action on $\cB_k$.

Let us now consider the unique principal $S^1$ bundle $\widetilde{\pi}:\mathcal{N}_k\stackrel{S^1}{\longrightarrow} \mathcal{B}_k$ whose Chern class is $[\omega]$. Using \eqref{e:omega_def}, we have that $\widetilde{\pi}^{-1}(B_r(p^{\alpha}_{\ell})\setminus\{p^{\alpha}_{\ell}\})$ is diffeomorphic to a punctured $4$-ball for each $p^{\alpha}_{\ell}$.  Hence $\mathcal{N}_k$ can be ``completed'' 
\begin{align}
M^4_k \equiv \cN_k\cup \{\tilde p^\alpha_\ell\}\, ,	
\end{align}
to a closed $4$-manifold $M^4_k$ by adding $2k$ points. Moreover, this completion gives rise to a smooth extension 
\begin{align}
	\pi:M^4_k\to S^3\, ,
\end{align}
of $\widetilde{\pi}:\mathcal{N}_k\to \mathcal{B}_k$. The $S^1$ action on $\mathcal{N}_k$ extends smoothly to $M_k$ with the added points as fixed points. See Section \ref{ss:lemma_smoothness_proof} for coordinate computations and also \cite{AndersonKronheimerLebrun,LeBrun}, \cite[Section 2]{Gross} for analogous arguments.
\medskip

Note that $[\omega]$ is a primitive class. Since $ \mathcal{B}_k$ is simply connected we have that $\mathcal{N}_k$ is simply connected, and hence $M_k$ is simply connected as well.

\medskip

We are going to define a metric on $M^4_k$ through a variant of the so-called Gibbons-Hawking ansatz, originally introduced in \cite{GibbonsHawking}.\\

Let $u:S^3\to [-\infty, +\infty]$ be a solution of
\begin{equation}\label{eq:defu}
	- \Delta_{S^3} u = 2\pi \sum_{\ell = 0}^{k-1} \left(\delta_{(e^{2\pi i \ell/k},0)} - \delta_{(0,e^{2\pi i \ell/k})}\right)= 2\pi \sum_{\ell = 0}^{k-1} \left(\delta_{p^0_\ell} - \delta_{p^1_\ell}\right)\, ,
\end{equation}
where $S^3$ is endowed with the round metric of radius one. As the right hand side integrates to zero we have that such a function $u$ exists, is smooth away from $\{p^\alpha_\ell\}$, and is unique up to a constant.  We normalize it in such a way that it is odd under the involution $\iota:S^3\to S^3$, $\iota(z_1,z_2):=(z_2,z_1)$.  That is we want the condition $u(z_2,z_1)=-u(z_1,z_2)$.  
Notice that $u: \mathcal{B}_k \to \R$ is smooth and harmonic. We can define the two-form
\begin{equation}\label{eq:curv2form}
	\omega = \ast du  \, ,
\end{equation}
where $\ast$ denotes the Hodge $\ast$ operator on $S^3$. Note that $d\omega=0$ on $\cB_k$ follows because $u$ is harmonic on $\cB_k$, and hence $\omega$ defines a 2-cohomology class on $\cB_k$.

\begin{lemma}\label{lemma:rightrep}
The cohomology class $[\omega]$ satisfies \eqref{e:omega_def}, and thus represents the Chern class of the principal $S^1$ bundle $\widetilde{\pi}:\mathcal{N}_k\to \mathcal{B}_k $ in de Rham cohomology.
\end{lemma}

\begin{proof}
Clearly, $u$ is invariant under the $\mathbb{Z}/k\mathbb{Z} \times \mathbb{Z}/k\mathbb{Z}$-action introduced in \eqref{eq:gamma_1gamma_2}. Hence, $\omega = \ast du$ is also invariant with respect to this action.

Let $B_r(p^\alpha_\ell)$ be a small three-ball centered at one of the removed points $p^\alpha_\ell$ with $S^\alpha_\ell=\partial B^3_r(p^\alpha_\ell)$.  Observe that
\begin{equation}\label{eq:domega}
	d \omega = d \ast d u = \ast (\ast d \ast d u) = \ast \Delta u\, .
\end{equation}
Hence, by (a distributional formulation of) Stokes' theorem, we get
\begin{equation}
	\int_{S^\alpha_\ell} \omega 
	= \int_{B_r(p^\alpha_\ell)} d \omega
	= \int_{B_r(p^\alpha_\ell)} \Delta u
	= (-1)^\alpha\,2\pi \, ,
\end{equation}
as claimed.\\
\end{proof}

Thanks to Lemma \ref{lemma:rightrep}, there exists a connection on the principal $S^1$-bundle $\widetilde{\pi}: \mathcal{N}_k \to \mathcal{B}_k$ with curvature form $\omega$. Let $\theta \in \Omega^1(\mathcal{N}_k)$ be a principal connection $1$-form corresponding to this connection. In particular, $d\theta = \widetilde{\pi}^*\omega$. While such a connection $1$-form is not unique, its uniqueness is ensured up to gauge equivalence as the base space is simply connected.

\medskip

Let us now consider the function
\begin{equation}\label{eq:V}
	V(x) : = \frac{1}{4\pi} \log\left(e^{4\pi x} + e^{-4\pi x} + 2\right) \, .
\end{equation}
We introduce the Riemannian metric $h_{k}$ on $\mathcal{N}_k$ by defining
\begin{equation}\label{eq:h_k}
h_{k} := 
V_\Lambda(u_k)\, \widetilde{\pi}^*g_{S^3}
+\frac{1}{k^2\Lambda^2\, V_\Lambda\left(u_k\right)}\theta\otimes\theta\, ,
\end{equation}
where 
\begin{align}
	u_k&:=u/k\, ,\notag\\
	V_\Lambda(x) &:= \Lambda^{-1}V(x) + 1\, ,
\end{align}
with $\Lambda>1$ to be chosen later to be sufficiently large. The introduction of the parameter $\Lambda$ is somewhat motivated by the expression for the Taub-NUT metrics via the Gibbons-Hawking ansatz.

Equivalently, $h_{k}$ is the unique principal $S^1$ bundle metric on the total space of $\widetilde{\pi}:\mathcal{N}_k\to \mathcal{B}_k$ such that: 
\begin{enumerate}
	\item The connection $1$-form is $\theta$ with $d\theta=\widetilde{\pi}^*\omega$;  This connection is smooth away from $\{p^\alpha_\ell\}$ but is singular in the natural $S^3$ coordinates near $\{p^\alpha_\ell\}$.
	\item The induced metric on the base space $\mathcal{B}_k$ is $V_\Lambda(u_k)\, g_{S^3}$;  For large $\Lambda$ the geometry of $\cB_k$ is close to $S^3$ away from $\{p^\alpha_\ell\}$, with singularities near  $\{p^\alpha_\ell\}$.
	\item The $S^1$ fibers have length $(k\Lambda \sqrt{V_\Lambda(u_k)})^{-1}$; For large $\Lambda$ the fibers are $\approx \frac{1}{k\Lambda}$ small away from $\{p^\alpha_\ell\}$, with fiber length tending to zero at $\{p^\alpha_\ell\}$. 
\end{enumerate}

\begin{remark}[Choice of $V$]
	The specific choice of $V$ in \eqref{eq:V} does not play a central role in the sequel; it is merely convenient for calculations. In principle, any smooth convex function $V: \mathbb{R} \to (0, +\infty)$, independent of $k$, with the correct asymptotic behavior $V(x) \sim |x|$ as $x \to \pm \infty$, would work.
\end{remark}

Despite the obvious singularities of the metric components in the $S^3$ coordinates, we claim that $h_k$ extends to a smooth Riemannian metric on $M^4_k$.  This Riemannian metric is invariant under an effective isometric action of the finite Heisenberg group $H_3(\mathbb{Z}/k\mathbb{Z})$. Moreover, the manifolds $(M^4_k,h_k)$ have uniformly bounded diameters, and $0 \le {\rm Ric}_{h_k} \le 3$ for every $k\in\N$. We will prove these claims, corresponding to Lemma \ref{lemma:smoothness}, Lemma \ref{lemma:invariance}, and Lemma \ref{lemma:Ricci} below respectively, in the forthcoming subsections. 

\begin{lemma}\label{lemma:smoothness}
The metric $h_k$ on $\mathcal{N}_k$ extends smoothly to a Riemannian metric, still denoted $h_k$, on $M^4_k$. 
\end{lemma}

\begin{lemma}\label{lemma:invariance}
The finite Heisenberg group $H_3(\mathbb{Z}/k\mathbb{Z})$ acts effectively and isometrically on $(M^4_k, h_k)$. 
\end{lemma}

\begin{lemma}\label{lemma:Ricci}
	The manifold $(M^4_k, h_k)$ satisfies
\begin{enumerate}
	\item $\mathrm{diam}(M_k) \leq \pi+O(k^{-1})$\, ,
	\item  $0 \leq \mathrm{Ric}_{h_k} \leq 2+O(k^{-1})$\, .
\end{enumerate}
Moreover, we have that $\mathrm{Ric}_{h_k} > 0$ in $ \mathcal{N}_k \subset M_k$.  That is, the Ricci tensor is zero only at the $2k$ points $\{\tilde p^\alpha_\ell\}$ added to $\mathcal{N}_k$ to form $M^4_k$.
\end{lemma}

Once the proofs of Lemma \ref{lemma:smoothness}, Lemma \ref{lemma:invariance}, and Lemma \ref{lemma:Ricci} are completed, to conclude the proof of Theorem \ref{thm:HeisenberIntro} we will perturb the geometry of $(M^4_k, h_k)$ with a smooth conformal factor $g_k := e^{-2 \phi_k} h_k$.   This will allow us to achieve ${\rm Ric}_{g_k} > 0$ while maintaining the Ricci and diameter upper bounds, as well as the isometric action of the Heisenberg group $H_3(\mathbb{Z}/k\mathbb{Z})$.

\begin{lemma}\label{lemma:Ricci2}
	There exists $\phi_k \in C^\infty(M^4_k)$ such that $|\phi_k| \leq O(k^{-1})$ and $\phi_k$ is invariant under the action of $H_3(\mathbb{Z}/k\mathbb{Z})$. Moreover, the metric $g_k := e^{- 2 \phi_k} h_k$ satisfies $0 < {\rm Ric}_{g_k} \leq 2+O(k^{-1})$.
\end{lemma}

Similar conformal perturbation arguments have been used in several other instances, for instance in \cite{Aubin} and \cite{Ehrlich}.

\begin{remark}
The $L^2$-norm of the Riemann tensor of $(M^4_k, g_k)$ is not uniformly bounded with respect to $k$. Indeed, by Chern-Gauss-Bonnet
\begin{equation}\label{eq:ChernGaussBonnet}
	\int_{M_k^4} |{\rm Riem}_{g_k}|^2 = 32 \pi^2 \chi(M^4_k) + \int_{M_k^4} 4 |{\rm Ric}_{g_k}|^2 - R^2_{g_k}\, .
\end{equation}
Although the Ricci and scalar curvature terms are bounded, $\chi(M^4_k) = 2k+1 \to \infty$ as $k \to \infty$. One can also verify that the sectional curvatures are not uniformly bounded, either from below or above.\\
\end{remark}

\subsection{Proof of Lemma \ref{lemma:smoothness}}\label{ss:lemma_smoothness_proof}

Some variants of the proof below have already appeared in the literature in slightly different contexts, see for instance \cite{AndersonKronheimerLebrun,LeBrun}.  It is also clearly enough to check that $h_k$ admits a smooth extension within each $\widetilde \pi^{-1}(B_{r}(p^\alpha_\ell))$ for some sufficiently small $0<r=r(k)<1$.\\

Let us first write each $B_{r}(p_\ell^\alpha)\subset S^3$ in polar coordinates as
\begin{align}
	&B_{r}(p_\ell^\alpha)\setminus\{p_\ell^\alpha\} \approx B_r(0^3)\setminus \{0^3\} \approx (0,r)\times S^2\, ,\notag\\
	&g_{S^3} = dt^2+\sin^2(t) g_{S^2}\, .\label{eq:gS3}
\end{align}

Recall that $\widetilde{\pi}^{-1}(B_{r}(p_\ell^\alpha)\setminus\{p_\ell^\alpha\})\cup \{\tilde p^\alpha_\ell\}\approx \dR^4$ is identified with a ball in Euclidean space.  We can make this identification a bit more explicitly as follows.  Observe that $[\omega]$, the defining cohomology class of the circle bundle, restricts to that of a Hopf bundle over each $S^2$ by \eqref{e:omega_def}.  We can write $\widetilde{\pi}^{-1}(B_{r}(p_\ell^\alpha))\approx \dR^4\setminus\{0^4\}$ in polar coordinates, as in \cite[Section 5, page 231]{LeBrun}, as the projection mapping
\begin{align}\label{eq:Projpi}
\pi:(0,r) \times S^3 \ni (s, x) \mapsto \left(\frac{s^2}{2}, \pi_{\rm Hopf}(x)\right) \in (0,r) \times S^2\, .
\end{align}

We see that in these coordinates $\pi$ extend smoothly to $\R^4$.  
Recall that the metric \eqref{eq:h_k} is written
\begin{equation}
	\begin{split}
	h_k & := \frac{1}{k\Lambda} (kV(u/k) + k\Lambda) \left( \pi^* g_{S^3} + \frac{1}{(kV(u/k) + k\Lambda)^2} \theta \otimes \theta\right)\, ,
	\\
	g_{S^3} & := dt^2 + \sin^2(t) g_{S^2} \, .
	\end{split}
\end{equation}
We have that $u$ is a solution of $- \Delta_{g_{S^3}} u = 2\pi (-1)^\alpha \delta_{0^3}$ , which we can write in these coordinates as 
\begin{align}\label{eq:uerror}
u = (-1)^\alpha\frac{1}{2}\frac{\cos(t)}{\sin(t)}+u'\, ,	
\end{align}
where $u'$ is a smooth solution of $\Delta_{S^3} u' = 0$ on $B_r(p^\alpha_\ell)$.
The connection $1$-form $\theta$ is determined by the equation $d\theta = \ast_{S^3} d u$ .  If $\theta_{\rm{Hopf}}$ is the Hopf connection $1$-form on $S^3$, then $(-1)^{\alpha}d\theta_{\rm{Hopf}}=d\theta-d\ast_{S^3}d u'$. Hence, up to gauge transformation, we have that 
\begin{align}\label{eq:thetavshopf}
	\theta = (-1)^\alpha\theta_{\rm{Hopf}}+\theta'\, ,
\end{align}
where $\theta'$ is smooth on $\widetilde{\pi}^{-1}(B_r(p^\alpha_\ell))\approx \dR^4$.

\medskip

Using the explicit polar coordinate expression of the projection $\pi: \R^4 \to \R^3$ from \eqref{eq:Projpi} and the expression for $g_{S^3}$ from \eqref{eq:gS3}, we can express 
\begin{equation}\label{exp1} 
	\pi^{*}g_{S^3} = s^2 \, ds^2 + \sin^2\left(\frac{s^2}{2}\right) \pi_{\rm Hopf}^{*}g_{S^2} \, .
\end{equation}  
On the other hand, from \eqref{eq:uerror} and the explicit expression of $V$, we can deduce that
\begin{equation}\label{exp2} 
	kV(u/k) + k\Lambda  
	= \frac{1}{2} \frac{\cos(t)}{\sin(t)} + u'' \, , 
	\quad  u'' \in C^\infty(B_{r_0}(0^3)) \, .
\end{equation}
Thus, using that $t=s^2/2$ and a Taylor expansion as $s\to 0$, 
\begin{equation}\label{eq:Taylorpotential}
kV(u/k) + k\Lambda=\frac{1}{s^2}+O(1)\, .
\end{equation}
By combining \eqref{eq:thetavshopf}, \eqref{exp1}, and \eqref{eq:Taylorpotential}, we obtain the expansion
\begin{equation}
	\begin{split}
		h_k &= \frac{1}{k\Lambda} \left( ds^2 + s^2 \pi_{\rm Hopf}^* \frac{1}{4}g_{S^2} +  s^2 \theta_{\rm Hopf} \otimes \theta_{\rm Hopf} \right) + h'
	     \\&=
		\frac{1}{k\Lambda} g_{\R^4} + h' \, ,
	\end{split}
\end{equation}
where $h' = O(r^2)$ is smooth. In particular, $h_k$ admits a smooth extension, as claimed. $\qed$\\

\subsection{Proof of Lemma \ref{lemma:invariance}}

It suffices to show that $H_3(\mathbb{Z}/k\mathbb{Z})$ acts effectively and isometrically on $(\mathcal{N}_k, h_k)$, as $(M_k, h_k)$ is the metric completion of $(\mathcal{N}_k, h_k)$.  

\medskip

Let $K\cong \mathbb{Z}/k\mathbb{Z}\times \mathbb{Z}/k\mathbb{Z}$ be the group of diffeomorphisms of $\mathcal{B}_k$ introduced in \eqref{eq:gamma_1gamma_2}. Let $H$ be the group of diffeomorphisms $\phi:\mathcal{N}_k\to \mathcal{N}_k$ such that:
\begin{itemize}
\item[(i)] $\phi$ is equivariant with respect to the $S^1$ action on $\mathcal{N}_k$, i.e., $R_{\theta}\circ\phi=\phi\circ R_{\theta}$ for every $\theta\in S^1$. Here we denoted by $R_{\theta}:\mathcal{N}_k\to\mathcal{N}_k$ the $S^1$ action on the principal $S^1$ bundle $\widetilde{\pi}:\mathcal{N}_k \to \mathcal{B}_k$;
\item[(ii)] $\phi$ is a lift of some element $\widetilde{\phi}\in K$, i.e., $\widetilde{\pi}\circ\phi=\widetilde{\phi}\circ\widetilde{\pi}$;
\item[(iii)] $\phi^*\theta=\theta$, where $\theta$ is the connection $1$-form on $\widetilde{\pi}:\mathcal{N}_k\to\mathcal{B}_k$.  
\end{itemize} 
\medskip
\noindent
\textbf{Claim 1:} $H$ acts isometrically on $(\mathcal{N}_k,h_k)$. 
\smallskip

For any $\phi\in H$ we can compute
\begin{align}
\phi^*h_k=&\, \phi^*\left(V_\Lambda(u_k)\, \widetilde{\pi}^*g_{S^3}
+\frac{1}{k^2\Lambda^2\, V_\Lambda\left(u_k\right)}\theta\otimes\theta\right)\\
=&\, V_\Lambda(u_k\circ\widetilde{\phi})\,\widetilde{\pi}^*\widetilde{\phi}^*g_{S^3}+\frac{1}{k^2\Lambda^2\, V_\Lambda\left(u_k\circ\widetilde{\phi}\right)}\phi^*\theta\otimes\phi^*\theta\\
=&\, V_\Lambda(u_k)\, \widetilde{\pi}^*g_{S^3}
+\frac{1}{k^2\Lambda^2\, V_\Lambda\left(u_k\right)}\theta\otimes\theta
=\, h_k\, ,
\end{align}
where we used (ii), (iii), and the fact that $K$ acts isometrically with respect to the standard metric on $S^3$.
\medskip

\noindent
\textbf{Claim 2:} Any $\widetilde{\phi}\in K$ admits a lift $\phi\in H$. Such a lift is unique modulo composition with the action $R_{\theta}$.
\smallskip

The Claim follows from a fairly standard lifting argument, see for instance \cite[Proposition 2.7]{HondaViaclovsky}.  Indeed, as $\widetilde{\phi}^*\omega=\omega$ and $H^2(\mathcal{B}_k,\mathbb{Z})$ is torsion-free we have that $\widetilde{\phi}$ preserves the Chern class of our $S^1$ bundle and hence lifts to an equivariant diffeomorphism $\hat\phi:\cN_k\to\cN_k$ .  This lift is well defined up to a gauge transformation, so let us observe that $d(\hat\phi^*\theta-\theta)=\omega-\omega=0$.  As $H^1(\mathcal{B}_k,\mathbb{Z})=0$ we can write $d(\hat\phi^*\theta-\theta)=-df$ with $f:\cN_k\to \dR$. We use $f$ to define our gauge tranformation we obtain $\phi = \hat\phi\cdot e^{i f}$, where $\phi^*\theta-\theta = \hat\phi^*\theta+df-\theta=0$, as claimed.  
\medskip

\noindent
\textbf{Claim 3:} $H$ contains a subgroup isomorphic to $H_3(\mathbb{Z}/k\mathbb{Z})$.
\smallskip

The restriction of $\omega$ to the Clifford torus $\Sigma:=\{|z_1|=|z_2|\}\subset \mathcal{B}_k$ is cohomologous to the $2\pi k$-multiple of the volume form. Indeed, up to a choice of orientation, from \eqref{eq:domega} we deduce that
\begin{equation}
	\int_\Sigma \omega = \int_{\{|z_1| < |z_2|\}} \Delta_{S^3} u = 2\pi k \, .
\end{equation}
Hence $\widetilde{\pi}: \mathcal{N}_k\to \mathcal{B}_k$ restricts to a principal $S^1$ bundle over the Clifford torus isomorphic to the degree $k$ Heisenberg nilmanifold ${\rm Nil}_k^3$.\\
Note that $K$ leaves the Clifford torus $\Sigma$ invariant. Hence by (ii) we can restrict the action of $H$ on $\mathcal{N}_k$ to an isometric action on a Riemannian Heisenberg nilmanifold ${\rm Nil}_k^3$ respecting the principal $S^1$ bundle structure. As the nilmanifold ${\rm Nil}_k^3$ is codimension one and the isometries of $H$ are orientation preserving, we see that the restriction homomorphism $r:H\to\mathrm{Iso}({\rm Nil}_k^3)$ is injective.  Moreover, we see it is an isomorphism onto the subgroup of isometries of ${\rm Nil}_k^3$ which are lifts of the the $\mathbb{Z}/k\mathbb{Z}\times \mathbb{Z}/k\mathbb{Z}$ action on $\Sigma\approx T^2$, viewed as the base of $S^1\to {\rm Nil}_k^3\to T^2$.  It then follows from the construction of Section \ref{sec:HeisenbergNil} that $r(H)$ (and hence $H$) has a subgroup isomorphic to $H_3(\mathbb{Z}/k\mathbb{Z})$ acting as in \eqref{eq:widetilderho}.  $\qed$ \\

\subsection{Proof of Lemma \ref{lemma:Ricci}}

We will prove first that for each $\eps>0$ if $\Lambda=\Lambda(\eps) \geq 1$ is sufficiently large, actually independent of $k$, then ${\rm diam}(M_k) \leq \pi+\eps$.  In particular, if $\Lambda(k)\to \infty$ then we have the claimed diameter bound on $M^4_k$.  

First observe that if $\Lambda\geq \Lambda(\delta)$ then $k^2 \Lambda^2 V_\Lambda \geq \delta^{-2}$, which ensures that the fiber lengths of the principal bundle $\widetilde{\pi} : \mathcal{N}_k \to \mathcal{B}_k$ are less than $\delta$ for every $k \geq 1$.  Thus it is sufficient to establish a uniform diameter bound for the metrics $V_\Lambda(u_k)\, g_{S^3}$ on the base space.  We may also focus our attention on $\mathcal{B}_k$, as $M^4_k$ is the (metric) completion of $(\mathcal{N}_k, h_k)$.

\medskip
To prove the diameter bound on $V_\Lambda(u_k) g_{S^3}$ we need to estimate $u_k$.  Let us begin with the following:\\

{\bf Claim: }$\exists$ $C>0$, independent of $k$, so that $|u_k(z)| \le C + |z_1|^{-1} + |z_2|^{-1}\, .$\\

Recall that $u_1:S^3\to [-\infty,+\infty]$ is the unique solution of 
\begin{equation}\label{eq:defu1}
	- \Delta_{S^3} u_1 = 2\pi \left(\delta_{(1,0)} - \delta_{(0,1)}\right)\, ,
\end{equation}
such that $u_1\circ\iota=-u_1$.  Note that we can also identify $u_1$ as the difference between the Green's functions on $S^3$ at $(1,0)$ and $(0,1)$.  Using the standard estimates on Green's functions in dimension three we have for $(z_1,z_2)\in S^3\subseteq \dC\times \dC$ the estimate
\begin{align}
	\Big|u_1(z_1,z_2)-\frac{1}{\sqrt{|z_1-1|^2+|z_2|^2}}+\frac{1}{\sqrt{|z_1|^2+|z_2-1|^2}}\Big|\leq C\, .
\end{align}
As a consequence we have the (very poor) estimate
\begin{align}\label{eq:roughest1}
	|u_1(z)|\leq C+|z_2|^{-1}+|z_1|^{-1}\, ,
\end{align}
which bounds $u_1$ uniformly away from the singular fibers $F_0$ and $F_1$.
For each $k\ge 2$ we can write
\begin{equation}\label{eq:iduk}
u_k(z_1,z_2)=\frac{1}{k}\sum_{\ell=0}^{k-1}u_1(e^{\frac{2\pi i\ell}{k}}z_1,e^{\frac{2\pi i\ell}{k}}z_2)\, .
\end{equation}

In particular we arrive at the (less poor, as $k$ gets large) estimate 
\begin{equation}\label{eq:roughest2}
	|u_k(z)| \le C + |z_1|^{-1} + |z_2|^{-1}\, ,
\end{equation}
which finishes the proof of the claim.\\

With the claim in hand we can now estimate $V_\Lambda(u_k):= \Lambda^{-1}V(u_k)+1$. To begin, we can directly plug in the estimate of the claim to get
\begin{equation}\label{eq:roughest3}
	\begin{split}
	V_\Lambda(u_k(z_1,z_2)) &\le \Lambda^{-1}(C + |z_1|^{-1} + |z_2|^{-1}) + 1 \,  ,
	\end{split}	
\end{equation}
for every $(z_1,z_2)\in S^3$.  Consider for $r>0$ the set $S^3\setminus B^{S^3}_{r}(F_0\cup F_1)$.  Note we have subtracted the $r$-tube with respect to the $S^3$ metric around each fiber off. Then using \eqref{eq:roughest3} we can choose $\Lambda\geq \Lambda(C,r,\delta)>0$ sufficiently large that
\begin{align}\label{eq:VLambdabound1}
	|V_\Lambda(u_k)-1|\leq \begin{cases}
 	\delta^3\left(1+\frac{1}{|z_1|} + \frac{1}{|z_2|} \right) &\text{ in } B^{S^3}_{r}(F_0\cup F_1)\\
 	\delta^3 &\text{ in }S^3\setminus B^{S^3}_{r}(F_0\cup F_1)\, .
 \end{cases}
\end{align}

In particular, the metric $V_\Lambda(u_k) g_{S^3}$ is uniformly close to the $S^3$ metric on the set $S^3\setminus B^{S^3}_{r}(F_0\cup F_1)$, and it is then clear that the diameter of this set is $\leq (1+\delta)\pi$.  To estimate the diameter of $\cB_k$ we are left showing that each point of $B^{S^3}_{r}(F_0)$ (and $B^{S^3}_{r}(F_1)$) can be connected to the boundary by a curve of small length. To this end we have the following, which focuses on $F_0$:\\

\textbf{Claim:} Let $z_1\in \C$ be such that $|z_1|=1$. Consider the curve $\gamma_{z_1}:(0,2\sqrt{r})\to S^3$ defined as
\begin{equation}
	\gamma_{z_1}(t):=(\sqrt{1-t^2}\,z_1,t^2)\, .
\end{equation}
If $\ell_k(\gamma_{z_1})$ denotes the length of $\gamma_{z_1}$ with respect to the metric $V_\Lambda(u_k)\, g_{S^3}$, then
\begin{equation}
	\ell_k(\gamma_{z_1}) \le 6r+24\delta^{3/2}\sqrt{r}\, .
\end{equation}

\medskip

To prove the claim, let us use \eqref{eq:VLambdabound1} and the estimate $|\gamma'_{z_1}(t)|\le 3t$ to obtain 
\begin{align}
\ell_k(\gamma_{z_1})
=&
\int_{0}^{2\sqrt{r}}\sqrt{V_\Lambda(u_k(\gamma_{z_1}(t)))}|\gamma'_{z_1}(t)| dt\\
\le & \,\int_{0}^{2\sqrt r}\left(1+4\frac{\delta^{3/2}}{t}\,\right)\, \cdot 3t\, dt\\
\le & \, 6r+24\delta^{3/2}\sqrt{r}\, ,
\end{align}

as claimed.\\

If we now choose $r=r(\eps)>0$, $\delta(\eps)>0$ sufficiently small with $\Lambda>\Lambda(C,r,\eps,\delta)>0$ sufficiently big, we can combine the estimates of this subsection to arrive at the diameter bound $\text{diam}(M^4_k)\leq \pi+\eps$, as claimed. $\qed$\\

\subsection{Proof of Lemma \ref{lemma:Ricci}, Ricci Curvature Bound}

Let $U$ and $X$ be a unit vertical vector and a unit horizontal vector for the $S^1$ bundle metric $h_{k}$ on $\widetilde{\pi}:\mathcal{N}_k\to \mathcal{B}_k$ respectively. We claim that
\begin{align}
\label{eq:vertfin}	{\rm Ric}_{h_k}(U,U)  =&\, \frac{ |d u_k|^2}{2} \frac{\Lambda}{V(u_k) + \Lambda}\left[\frac{V''(u_k)}{V(u_k) + \Lambda}+\frac{1-(V'(u_k))^2}{(V(u_k) + \Lambda)^2}\right] \, ,
	\\
\label{eq:mixedfin}	{\rm Ric}_{h_k}(U,X)  =& \,  0\, ,
	\\
\label{eq:basefin}	{\rm Ric}_{h_k}(X,X) 
	 = & \,\frac{\Lambda}{V(u_k) + \Lambda}\Big(2 +\frac{(1-(V'(u_k))^2)}{2(V(u_k)+\Lambda)}[du_k(X)]^2\\
	 \label{eq:basefin2}
	& - \frac{ |d u_k|^2}{2} \left[\frac{V''(u_k)}{V(u_k) + \Lambda}+\frac{1-(V'(u_k))^2}{(V(u_k) + \Lambda)^2}\right]\Big) ,
\end{align} 
where $|d u_k|$ is computed with respect to $g_{S^3}$, and with a slight abuse of notation, we denoted $du(X):=du(d\widetilde{\pi}(X))$. 

We postpone the verification of the expressions of the Ricci curvature to the end of the subsection and first exploit them to complete the proof of Lemma \ref{lemma:Ricci}.

\medskip

Consider the identities
\begin{equation}
	1 - (V'(x))^2 = \frac{4}{e^{4\pi x} + e^{-4\pi x} + 2} \, ,
	\quad
	V''(x) = \frac{8\pi}{e^{4\pi x} + e^{-4\pi x} + 2}\, .
\end{equation}
We can use the asymptotic expansion of $u_k$ near the poles to estimate
\begin{equation}
	\frac{|du_k|^2}{e^{4\pi u_k} + e^{-4\pi u_k}
	+ 2} \le C(k) \, .
\end{equation}
If we combine with the previous estimate we get 
\begin{align}
0&<|du_k|^2 V''(u_k)\leq C(k)\notag\\
0&< |du_k|^2(1-(V'(u_k))^2) \leq C(k)\, .
\end{align}
If we choose $\Lambda\geq \Lambda(k,\delta)$ we then get the estimates
\begin{align}
0&<|du_k|^2\frac{ V''(u_k)}{V(u_k)+\Lambda}\leq \delta\notag\\
0&< |du_k|^2\frac{(1-(V'(u_k))^2)}{V(u_k)+\Lambda} \leq \delta\, ,
\end{align}
as well as
\begin{align}
0\leq \frac{\Lambda}{V(u_k)+\Lambda} \leq 1\, ,	
\end{align}
with strict positivity near the poles.  If we plug these into our formulas for the Ricci curvature we arrive at the estimates (away from the poles of $u_k$)
\begin{align}
	0<&\, {\rm Ric}_{h_k}(U,U)\leq \delta\, ,\notag\\
	&\, {\rm Ric}_{h_k}(X,U) = 0\, ,\notag\\
	0<&\, {\rm Ric}_{h_k}(X,X)<2+10\delta\, ,
\end{align}
which proves our desired Ricci curvature bounds.

\bigskip
We are left with the verification of the expressions \eqref{eq:vertfin}, \eqref{eq:mixedfin}, and \eqref{eq:basefin}, for the Ricci curvatures of the metric $g_k$. To this aim, it is convenient to view
\begin{equation}\label{eq:useless}
h_{k}=\frac{1}{k \Lambda}\left[W_k(u)\, \widetilde{\pi}^*g_{S^3}+\frac{1}{W_k(u)}\theta\otimes\theta\right]\, ,
\end{equation}
for $W_k(x):=k(V(x/k) + \Lambda)$ for each $k\ge 1$, and compute the Ricci curvatures of a general Riemannian circle bundle metric of the form 
\begin{equation}
W\, \widetilde{\pi}^*g_{S^3}+\frac{1}{W}\theta\otimes\theta\, ,
\end{equation}
with $W:S^3\to (0,+\infty]$. If $\omega$ denotes the curvature $2$-form of such circle bundle, $U$ is a unit vertical vector and $X$ any horizontal vector, then there hold 
\begin{align}
\label{eq:RicUUgen}{\rm Ric}(U,U)  = &\, \frac{1}{2} W^{-2} \Delta_{g_{S^3}} W + \frac{1}{2} W^{-3}\left( |\omega|_{g_{S^3}}^2 - |\nabla W|_{g_{S^3}}^2 \right)\, ,
\\
\label{eq:RicUXgen}{\rm Ric}(U,X)  = &\,  -\frac{1}{2} W^{-3/2}\left(\delta_{g_{S^3}} \omega (X) - W^{-1} \omega(\nabla^{g_{S^3}} W,X)\right)\, ,
\\
\nonumber {\rm Ric}(X,X)  =&\, 2 |X|_{g_{S^3}}^2  - \frac{1}{2} \frac{\Delta_{g_{S^3}}W}{W} |X|_{g_{S^3}}^2
\\ 
\label{eq:RicXXgen}\quad \quad &\, - W^{-2}\frac{1}{2}\left(|\iota_X \omega|^2_{g_{S^3}} - |dW \wedge X^{\flat}|^2_{g_{S^3}}\right)\, .
\end{align}
To prove \eqref{eq:RicUUgen}, \eqref{eq:RicUXgen}, and  \eqref{eq:RicXXgen} we can start from the well-known formulas for the Ricci curvatures of Riemannian circle bundles $\pi:M\to B$ with fibres length $f$, curvature $2$-form $\omega$, and induced Riemannian metric on the base $g_B$. Namely 
\begin{align}
\label{eq:RicUUstd} {\rm Ric}(U,U) = &\,-\frac{\Delta_B f}{f} +\frac{f^2}{2}|\omega|^2_B\, ,\\
\label{eq:RicUXstd} {\rm Ric}(U,X)  = &\, \frac{1}{2}\left(-f\delta_B\omega(X)+3\omega(X,\nabla^Bf)\right),\\
\label{eq:RicXXstd} {\rm Ric}(X,X)  = & \, {\rm Ric}_B(X,X)-\frac{f^2}{2}|\omega(X)|_B^2-\frac{\nabla^2_Bf(X,X)}{f}\, ,
\end{align}
where as above $U$ denotes a unit vertical vector and $X$ denotes any horizontal vector. See for instance \cite[Proposition 9.36]{Besse}, or \cite{GilkeyParkTuschmann}. 

To get from \eqref{eq:RicUUstd} to \eqref{eq:RicUUgen} it suffices to exploit the formula for the transformation of the Laplacian under a conformal change in dimension $3$, i.e.,
\begin{equation}
\Delta_{Wg_{S^3}}f= W^{-1}\left(\Delta_{S^3}f+\frac{1}{2}W^{-1}g_{S^3}(\nabla^{g_{S^3}} W,\nabla^{g_{S^3}} f) \right)\, ,
\end{equation}
and the norm of $2$-forms
\begin{equation}
|\omega|^2_{Wg_{S^3}}=W^{-2}|\omega|^2_{g_{S^3}}\, .
\end{equation}
To get from \eqref{eq:RicUXstd} to \eqref{eq:RicUXgen} we rely on the formula for the transformation of the divergence for $2$-forms under conformal changes, i.e.,
\begin{equation}
\delta_{Wg_{S^3}}\omega=W^{-1}\left(\delta_{g_{S^3}}\omega+\frac{1}{2}W^{-1}\iota_{\nabla^{g_{S^3}} W}\omega\right)\, .
\end{equation}
To obtain \eqref{eq:RicXXgen} from \eqref{eq:RicXXstd} we combine the formulas for the transformation of Ricci curvature and Hessians under conformal changes to establish the identity
\begin{align}
 \nonumber {\rm Ric}_{Wg_{S^3}} - \frac{ {\rm Hess}_{Wg_{S^3}} W^{-1/2} }{W^{-1/2}} 
 =&\,  {\rm Ric}_{g_{S^3}} - \frac{1}{2} \left( \frac{\Delta_{g_{S^3}} W}{W} - \frac{|dW|^2}{W^2} \right) g_{S^3}\\ 
 &- 2 \frac{dW}{W} \otimes \frac{dW}{W}\, .
\end{align}
\medskip

Under the assumption that $\omega=\ast du$, where $u:\mathcal{B}_k \to \R$ is harmonic with respect to the round metric $g_{S^3}$, \eqref{eq:RicUUgen}, \eqref{eq:RicUXgen}, and  \eqref{eq:RicXXgen} simplify into 
\begin{align}
\label{eq:RicUU1}	{\rm Ric}(U,U) & = \frac{1}{2} \frac{\Delta W}{W^2} + \frac{1}{2} \frac{|du|^2 - |dW|^2}{W^3}\, ,
	\\
\label{eq:RicUX1}	{\rm Ric}(U,X) & = \frac{1}{2} \frac{\ast(du \wedge dW)(X)}{W^{5/2}}\, ,
	\\
\label{eq:RicXX1}	{\rm Ric}(X,X) & =  2 |X|^2 - \frac{1}{2} \frac{\Delta W}{W} |X|^2
	- \frac{1}{2}
	\frac{|du \wedge X^{\flat}|^2 - |dW \wedge X^{\flat}|^2}{W^2}\, .
\end{align}
Above, it is understood that norms and Laplacians are computed with respect to the round metric $g_{S^3}$. 
If we also assume that, with a slight abuse of notation, $W=W\circ u$, then we can rewrite \eqref{eq:RicUU1}, \eqref{eq:RicUX1}, and \eqref{eq:RicXX1} as
\begin{align}
\label{eq:RicUUalm}	{\rm Ric}(U,U)  = &\, \frac{1}{2} \frac{|du|^2}{W} 
	\left( \frac{W''}{W} + \frac{1 - (W')^2}{W^2}  \right) \, ,
	\\
\label{eq:RicUXalm}	{\rm Ric}(U,X)  = &\,  0\, ,
	\\
\label{eq:RicXXalm}	{\rm Ric}(X,X) = &\,  2|X|^2+ \frac{1 - (W')^2}{2W^2} (du(X))^2\\
\nonumber	&\, - \frac{1}{2}\left( \frac{W''}{W} + \frac{1 - (W')^2}{W^2}  \right) |du|^2 |X|^2\, .
\end{align}
The expressions \eqref{eq:vertfin}, \eqref{eq:mixedfin}, and \eqref{eq:basefin} can be easily obtained from \eqref{eq:RicUUalm}, \eqref{eq:RicUXalm}, and \eqref{eq:RicXXalm}, recalling that $W_k(x):=k(V(x/k) + \Lambda)$ and scaling to account for the $\frac{1}{k\Lambda}$ factor in \eqref{eq:useless}. $\qed$

\subsection{Proof of Lemma \ref{lemma:Ricci2}}
Recall that $M_k \setminus \mathcal{N}_k$ is a set of $2k$ points forming an orbit of the isometric action of the Heisenberg group $H_3(\mathbb{Z}/k\mathbb{Z})$. We define $\rho_k(x):= \frac{1}{2}d_{h_k}(x, M_k \setminus \mathcal{N}_k)^2$, the square of the distance function to this set with respect to the metric $h_k$. Note that $\rho_k$ is globally $H_3(\mathbb{Z}/k\mathbb{Z})$-invariant and smooth in a sufficiently small neighborhood of $M_k \setminus \mathcal{N}_k$. Let $\eta_k \in C^\infty(\mathbb{R})$ be a smooth function such that $\eta_k(t) = t$ for $t \leq r_k$ and $\eta_k(t) = 0$ for $t \geq 2r_k$, where $r_k > 0$ is sufficiently small to ensure that $\eta_k(\rho_k) \in C^\infty(M_k)$ and
\begin{equation}\label{eq:rhok}
{\rm Hess}_{h_k} \rho_k \ge \frac{1}{2}h_k \, , 
\quad 
|d \rho_k |_{h_k} \le 10^{-2} \, , \quad \text{when $\rho_k \le r^2_k$} \, .
\end{equation}
The existence of such $r_k$ follows from the fact that
\begin{equation}
{\rm Hess}_{h_k}\rho_k= h_k\, ,\quad d \rho_k=0\, ,
\end{equation}
for every ${\tilde p^\alpha_\ell}\in M_k\setminus\mathcal{N}_k$.

 Finally, we define $\phi_k := \eps_k \eta_k(\rho_k)$, where $\eps_k > 0$ is a small parameter to be chosen later. The Ricci curvature of $g_k:= e^{- 2 \phi_k} h_k$ is given by
\begin{equation}
{\rm Ric}_{g_k} = {\rm Ric}_{h_k} + 2 {\rm Hess}_{h_k} \phi_k + 2 d\phi_k \otimes d \phi_k + ( \Delta_{h_k} \phi_k - 2| d\phi_k|_{h_k})h_k \, .
\end{equation}
It is clear that ${\rm Ric}_{g_k}>0$ in the set $\{\rho_k \le r_k/2\}$ as a consequence of \eqref{eq:rhok}. On the other hand, ${\rm Ric}_{h_k}\ge c>0$ in the set $\{\rho_k \ge r_k/2\} \subset \mathcal{N}_k$, hence ${\rm Ric}_{g_k}$ is positive in this region provided $\eps_k>0$ is small enough.
Similarly, we have the upper bound ${\rm Ric}_{g_k}\le {\rm Ric}_{h_k}+C\eps_k$. $\qed$

\section{Proof of Theorem \ref{thm:mainIntro}}\label{sec:Proofmainthm}

The Heisenberg actions in Theorem \ref{thm:HeisenberIntro} are not free, and thus do not yield the desired fundamental group. One can verify that the quotient of $M_k^4$ by the Heisenberg action $H_3(\mathbb{Z}/k\mathbb{Z})$ is homeomorphic to the four-sphere $S^4$.\\

To address this issue, consider the oriented frame bundle $\pi_{FM}:FM_k \stackrel{\mathrm{SO}(4)}{\longrightarrow}M^4_k$.  Given constants $d_k>0$ there is a unique principal bundle metric $g^{FM}_k$ on $FM_k$ defined by the conditions:
\begin{enumerate}
	\item The right $\mathrm{SO}(4)$ actions on $FM_k$ are isometric. More specifically, let $\{v^a\}$ be an orthonormal basis of $\mathrm{so}(4)$, the Lie algebra of $\mathrm{SO}(4)$, with respect to the standard Frobenius inner product. 
	   If $V^a$ denotes the corresponding pushforward vertical vector fields on $FM_k$, then the metric satisfies $g^{FM}_k(V^a, V^b) = d_k^2 \, \delta^{ab}$.

	\item Consider the natural connection form $\theta \in \Omega^1(FM_k; \mathrm{so}(4))$ associated with the Levi-Civita connection $\nabla_k$ on $(M_k, g_k)$.  If $H^j$ are vector fields on $M_k$ we may lift them to horizontal vector fields on $FM_k$, sloppily also denoted by $H_k$.  Then $g^{FM}_k(H^j,V^a)=0$.

	\item If $H^j$ are horizontal lifts of vector fields on $M_k$, then $g^{FM}_k(H^j,H^k) = g_k(H^j,H^k)$.
\end{enumerate}

We begin with the usual observation about isometries on $M_k$ and their lifts to $FM_k$. Namely, note that an isometry on $M_k$ maps an orthonormal frame to an orthonormal frame, and thus lifts to a mapping on $FM_k$.  It is easy to check this lifted mapping is an isometry with respect to a metric as above and an $\mathrm{SO}(4)$-bundle automorphism, i.e., it commutes with the right $\mathrm{SO}(4)$-action on $FM_k$.  Additionally, since an isometry fixes a point and a frame (e.g. its derivative is the identity) if and only if it is globally the identity, we have that the lifted isometric action on $FM_k$ is free.  In particular, we can lift our $H_3(\dZ/k\dZ)$ action to a {\it free} action on $FM_k$. 

To understand the Ricci curvature of $FM_k$, let us begin with the observation that if $d_k$ is very small then our $SO(4)$ fibers are very small.  In particular, they are very positively Ricci curved.  We want to see that the positive Ricci curvature of $M_k$ together with the potentially very positive Ricci curvature of the fibers leads to positive Ricci curvature of $FM_k$.

More precisely, let $V$ be a $g^{FM}$ unit vertical vector with $H$ unit a horizontal vector.  Note that the $V$ can be identified as a $2$-form in $M_k$ (in local coordinates  $\sum_{ij}\theta_{ij}(V) dx_i \wedge dx_j$), and so we can identify ${\rm Rm}_k[V]$ as a $2$-form.  Now recall the formulas (see \cite{ONeill} and also \cite{BerardBergery,Nash,Besse}):
\begin{align}
	{\rm Ric}^{FM}(V,V) &= d_k^{-2} + \frac{1}{4}|{\rm Rm}_k[V]|^2_k\geq d_k^{-2}\, ,\notag\\
	{\rm Ric}^{FM}(V,H) &= \,\langle \text{div} {\rm Rm}_k[H],V\rangle_{FM} \, ,\implies |{\rm Ric}^{FM}(V,H)|\leq Cd_k\, ,\notag\\
	{\rm Ric}^{FM}(H,H) &= {\rm Ric}_k(H,H)
	- \frac{3}{4}d_k^2 \sum_{i=1}^3 |{\rm Rm_k}(H,H_i)|^2_k\geq \epsilon_k-C d_k^2\, ,
\end{align}
where $H, H_1, H_2, H_3$ form an orthonormal frame of $(M_k,g)$, $C=C(n, {\rm Rm}_k, \nabla{\rm Rm}_k)$ and $\epsilon_k$ is the positive Ricci lower bound on $M^4_k$.  In particular, if $d_k$ is sufficiently small (depending on the size of ${\rm Rm}_k$ and $\nabla{\rm Rm}_k$ and the lower Ricci bound on $M^4_k$), then positive Ricci on $M^4_k$ implies positive Ricci on $FM_k$ .

We have thus built a ten-dimensional compact manifold with positive Ricci curvature and a free action of $H_3(\dZ/k\dZ)$.  However, we are not quite done yet.  As $M_k$ is homeomorphic to the connected sum of $(k-1)$ copies of $S^2\times S^2$ we have that $M_k$ is a spin manifold. Indeed, as noted in Section \ref{subsec:sinS1bundle}, the classification results in \cite{ChurchLamotke} yield the stronger statement that $M_k$ is diffeomorphic to the connected sum of $(k-1)$ copies of $S^2\times S^2$.
In particular, although $M_k$ is simply connected, its frame bundle $FM_k$ is not. More precisely, $\pi_1(FM_k)=\dZ/2\dZ$.
We therefore consider the spin frame bundle $P_{\rm Spin}M_k$, which is the double cover of $FM_k$. It is simply connected, and the action of $H_3(\dZ/k\dZ)$ on $FM_k$ lifts to an action of a $\dZ/2\dZ$-extension of $H_3(\dZ/k\dZ)$ on $P_{\rm Spin}M_k$. To simplify notation, we denote this extension by $G_k$. Note that every element $g\in H_3(\dZ/k\dZ)$ admits two lifts in $G_k$ which are $\mathrm{Spin}(4)$-bundle automorphisms and isometries on $P_{\rm Spin} M_k$. In particular, the $G_k$-action and the $\mathrm{Spin}(4)$ action do commute.

\subsection{Dimension $n=9$}
To reduce the dimension of $P_{\rm Spin} M_k$ from $10$ to $9$ and conclude the proof of Theorem \ref{thm:mainIntro}, we quotient $P_{\rm Spin} M_k$ by a suitably chosen closed subgroup $S^1<\mathrm{Spin}(4)$ and consider the induced action of $G_k$.

Since the actions of $G_k$ and $\mathrm{Spin}(4)$ commute, for any closed subgroup $H<\mathrm{Spin}(4)$, the quotient $P_{\rm Spin} M_k/H$, endowed with the quotient metric, carries a naturally induced isometric action of $G_k$.

We let $T^2<\mathrm{Spin}(4)\cong \mathrm{SU}(2)\times \mathrm{SU}(2)$ be the maximal torus described by
\[
T^2 = \left\{
\left(
\begin{pmatrix}
e^{i\theta_1} & 0 \\
0 & e^{-i\theta_1}
\end{pmatrix},
\begin{pmatrix}
e^{i\theta_2} & 0 \\
0 & e^{-i\theta_2}
\end{pmatrix}
\right)
\,:\,
\theta_1,\theta_2 \in \mathbb{R}
\right\}.
\]
Note that $T^2$ is totally geodesic with respect to any biinvariant metric on $\mathrm{Spin}(4)$. Thus, any closed circle subgroup $S^1\cong H<T^2$ is totally geodesic in $\mathrm{Spin}(4)$ as well. We deduce that, for any such $H<T^2$, the orbits of the induced (right-)$H$ action on $P_{\rm Spin} M_k/H$ are totally geodesic. By the formulas for the Ricci curvature of Riemannian circle bundles (see for instance \cite{Besse,GilkeyParkTuschmann}) we infer that $P_{\rm Spin} M_k/H$ has positive Ricci curvature. It is also straightforward to check that any such quotient $\widetilde N_k^9:=P_{\rm Spin} M_k/H$ is simply connected, by applying long exact sequence in homotopy to the circle bundles $S^1\to P_{\rm Spin} M_k\to \widetilde N^{9}_k$.  

To conclude the proof, it suffices to show the following:

\medskip
\noindent
{\bf Claim.} There exists a closed circle subgroup $H<T^2$ such that the induced $G_k$ action on $P_{\rm Spin} M_k/H$ is free. 
\medskip

Fix $p\in \widetilde N^{10}_k$. For every $g\in G_k$ which fixes the $\mathrm{Spin}(4)$-fiber containing $p$ (equivalently, such that the corresponding element $[g]\in H_3(\mathbb{Z}/k\mathbb{Z})$ fixes $\pi(p)$) there exists a unique element $k_g(p)\in \mathrm{Spin}(4)$ such that $g\cdot p=p\cdot k_g(p)$. 
Note that the map $g\to k_g(p)$ is a group homomorphism, in particular $k_g(p)$ has always finite order as $G_k$ is a finite group.

\begin{remark}
	Although this will not be used in the sequel, we note that if $p'$ lies in the same $\mathrm{Spin}(4)$-fiber as $p$, then $k_g(p)$ and $k_g(p')$ are conjugate. Indeed, let $A\in\mathrm{Spin}(4)$ be the unique element such that $p=p'\cdot A$. Then $k_g(p)=A^{-1}\cdot k_g(p')\cdot A$.
\end{remark}

For any closed subgroup $H<\mathrm{Spin}(4)$, the induced $G_k$-action on $P_{\rm Spin} M_k/H$ is free if and only if there do not exist $g\in G_k$ with $g\neq e$ and $p\in P_{\rm Spin} M_k$ as above with $k_g(p)\in H$. In order to find a closed circle subgroup $H<T^2$ with this property, it suffices to show that there are only finitely many $k_g(p)$ for some $g\in G_k$ and $p\in \widetilde N^{10}_k$ in $T^2$. 

To this end, observe that the eigenvalues of $k_g(p)$ are $\mathrm{Ord}(g)$-th roots of unity. Moreover, they depend continuously on $p$. It follows that the eigenvalues are constant on each connected component of $\pi^{-1}(\mathrm{Fix}([g]))\subset P_{\rm Spin} M_k$. The claim then follows because $G_k$ is finite, each fixed-point set $\mathrm{Fix}([g])\subset M^4_k$ has only finitely many connected components, and every $k_g(p)\in T^2$ is diagonal and hence uniquely determined by its eigenvalues up to permutation of the diagonal elements. $\qed$

\end{document}